\documentclass[12pt]{amsart}
\usepackage{amsmath, amssymb, amsfonts}
\usepackage{graphicx, overpic}
\usepackage{fullpage}
\usepackage{enumerate}
\usepackage{tikz} 
\usetikzlibrary{knots}
\usepackage{pst-knot}
\usepackage{epstopdf}
\usepackage{subfigure}
\usepackage[all]{xy}
\usepackage{extarrows}
\usetikzlibrary{patterns}
\usetikzlibrary{arrows}
\usepackage{abstract}
\usepackage{here}
\usepackage{todonotes}
\usepackage{enumitem}
\usepackage{microtype}
\usepackage{color}
\usepackage{xcolor}

\usepackage{hyperref}
\hypersetup{
    colorlinks=true,
    linkcolor=blue,
    filecolor=magenta,      
    urlcolor=cyan,
    }

\usetikzlibrary{decorations.pathreplacing,decorations.markings}
\usetikzlibrary{patterns}
\usetikzlibrary{arrows}
\usetikzlibrary{decorations.markings}

\theoremstyle{plain}
\newtheorem{thm}{Theorem}[section]
\newtheorem{lemma}[thm]{Lemma}
\newtheorem{prop}[thm]{Proposition}
\newtheorem{cor}[thm]{Corollary}
\newtheorem{definition}[thm]{Definition}
\newtheorem*{theorem*}{Theorem}
\newtheorem*{prop*}{Proposition}
\newtheorem*{main theorem*}{Main Theorem}

\theoremstyle{definition}
\newtheorem{remark}[thm]{Remark}
\newtheorem{example}[thm]{Example}

\numberwithin{equation}{section}
\numberwithin{figure}{section}

\tikzset{middlearrow/.style={
		decoration={markings,
			mark= at position 0.55 with {\arrow{#1}} ,
		},
		postaction={decorate}
	}
}

\DeclareMathOperator{\interior}{int}

\begin{document}

\title{Abelian Splittings and JSJ-Decompositions of Finitely Presented Bestvina--Brady Groups}
\author{Yu-Chan Chang}
\address{Department of Mathematics and Computer Science \\ Wesleyan University \\  Middletown, CT 06459}
\email{yuchanchang74321@gmail.com}

\maketitle

\begin{abstract}
We give a characterization of finitely presented Bestvina--Brady groups which split over abelian subgroups and describe the JSJ-decompositions of those Bestvina--Brady groups.
\end{abstract}

\section{Introduction}

Given a finite simplicial graph $\Gamma$, the associated \emph{right-angled Artin group (RAAG)} $A_{\Gamma}$ is generated by the vertex set $V(\Gamma)$ of $\Gamma$, and the relators are commutators $[v,w]$ whenever $(v,w)\in E(\Gamma)$. RAAGs have been extensively studied because they contain many interesting subgroups; see \cite{charneyanintroductiontorightangledartingroups} for an introduction to RAAGs. In this article, we study one of the subgroups of RAAGs known as \emph{Bestvina--Brady groups}. Let $\phi\colon A_{\Gamma}\to\mathbb{Z}$ be the group homomorphism which sends all the generators to $1$. The Bestvina--Brady group $BB_{\Gamma}$ is defined to be the kernel of $\phi$. Bestvina--Brady groups have an important connection to topology: there are Bestvina--Brady groups that are either counterexamples to the Eilenberg--Ganea Conjecture or Whitehead Conjecture; see \cite[Theorem 8.7]{bestvinabradymoresetheoryandfinitenesspropertiesofgroups}. 

We say that a group $G$ \emph{splits} over a subgroup $C$ if $G$ decomposes as an amalgamated product $G=A\ast_{C}B$ with $A\neq C$, $B\neq C$ or an HNN-extension $G=A\ast_{C}$. If the subgroup $C$ is abelian, then we way that $G$ has an \emph{abelian splitting}. Clay \cite{claywhendoesarightangledartingroupsplitoverz} showed that a RAAG splits over $\mathbb{Z}$ if and only if its defining graph has a cut-vertex. Groves and Hull generalized Clay's result to  abelian subgroups of higher rank. They proved that a RAAG splits over an abelian subgroup if and only if its defining graph contains a separating clique \cite{grovesandhullabeliansplittingsofraags}. Recently, Hull \cite{HullSplittingsofRAAGs} studied the splittings of RAAGs over non-abelian subgroups. In a more general setting, Barquinero, Ruffoni, and Ye \cite{BarquineroRuffoniYeGraphicalSplittingsofArtinKernels} studied the splittings of Artin kernels, which are generalizations of Bestvina--Brady groups. In this paper, we characterize the abelian splittings of finitely presented Bestvina--Brady groups:

\begin{theorem*}{(Theorem~\ref{abelian splittings of BBG})}
Let $\Gamma$ be a finite simplicial connected graph such that the flag complex on $\Gamma$ is simply connected. The Bestvina--Brady group $BB_{\Gamma}$ splits over an abelian subgroup if and only if $\Gamma$ satisfies one of the following:
\begin{enumerate}
\item $\Gamma$ a cut-vertex;
\item $\Gamma$ is a complete graph;
\item $\Gamma$ has a separating clique $K_{n}$, $n\geq2$.
\end{enumerate}
\end{theorem*}

The strategy of proving Theorem~\ref{abelian splittings of BBG} is similar to the one in \cite{grovesandhullabeliansplittingsofraags}. The main difference is that we focus on edges instead of vertices, since finitely generated Bestvina--Brady groups are generated by (directed) edges; see Theorem~\ref{Dicks--Leary presentations}. We want to point out that the cut-vertices and separating cliques of sizes greater or equal to two are distinguished in our result, but not in \cite{grovesandhullabeliansplittingsofraags}; see Remark~\ref{remark on distinguishing cut-vertices and cliques}.

Similar to the situation of RAAGs, some splittings of Bestvina--Brady groups can be seen directly from the defining graphs. If $\Gamma$ has a cut-vertex, then $BB_{\Gamma}$ splits as a free product. If $\Gamma$ is a complete graph on $n$ vertices, then $BB_{\Gamma}\cong\mathbb{Z}^{n-1}$ (see Example~\ref{example of abelian group for D-I presentation}) and splits as an HNN-extention. If $\Gamma$ contains a separating clique $K$ and $\Gamma\setminus K$ is a disjoint union of two induced subgraphs $\Gamma_{1}$ and $\Gamma_{2}$, then $BB_{\Gamma}$ splits as an amalgamated product $BB_{\Gamma_{1}\cup K}\ast_{BB_{K}}BB_{\Gamma_{2}\cup K}$. Note that $BB_{K}$ is a free abelian group. These observations rely on the \emph{Dicks--Leary presentation} for Bestvina--Brady groups; see Theorem~\ref{Dicks--Leary presentations}. 

A \emph{JSJ-decomposition} of a group $G$ is a graph of groups decomposition that encodes all the possible splittings of $G$ over a fixed family of subgroups. We refer the reader to \cite{guirardelandlevittjsjdecompositionsofgroups} for a comprehensive introduction to the JSJ theory of groups. For some groups, their JSJ-trees (the Bass--Serre tree of a JSJ-decomposition) is a quasi-isometry invariant. Bowditch \cite{bowditchcutpointsandcanonicalsplittingsofhyperbolicgroups} constructed a canonical JSJ-splitting of $1$-ended hyperbolic groups over $2$-ended subgroups, and the JSJ-tree for this splitting is a quasi-isometry invariant. Dani and Thomas \cite{danithomasbowditchsjsjtreeandthequasiisometryclassificationofcertaincoxetergroups} used Bowditch's JSJ-tree to classify certain hyperbolic right-angled Coxeter groups up to quasi-isometry. Given a finite simplicial connected graph $\Gamma$, the authors in \cite{grovesandhullabeliansplittingsofraags} studied an action of $A_{\Gamma}$ on a tree $T$ such that each vertex of $\Gamma$ acts elliptically with abelian edge stabilizers. They built a JSJ-decomposition for $A_{\Gamma}$ and called it a \emph{vertex-elliptic abelian JSJ-decomposition}. If we restrict the action of $A_{\Gamma}$ on $T$ to $BB_{\Gamma}$, then each edge of $\Gamma$ also acts on $T$ elliptically, and the edge stabilizers are still abelian groups. Borrowing the terminologies from \cite{grovesandhullabeliansplittingsofraags}, we describe an \emph{edge-elliptic abelian JSJ-decomposition} for $BB_{\Gamma}$; see Theorem~\ref{edge-elliptic abelian JSJ-decomposition of BBGs}.

The paper is organized as follows. In Section~\ref{section Preliminaries}, we give some preliminaries on Bestvina--Brady groups and groups acting on trees. In Section~\ref{section abelian splittings of bestvina--brady groups}, we prove Theorem~\ref{abelian splittings of BBG}. In Section~\ref{section JSJ-decomposition of bestvina--brady groups}, we give a JSJ-decomposition for the Bestvina--Brady groups described in Theorem ~\ref{abelian splittings of BBG}.

\section{Preliminaries}\label{section Preliminaries}

\subsection{Bestvina--Brady Groups}\label{subsection Bestvina--Brady groups}

Let $\Gamma$ be a finite simplicial graph. Recall that the Bestvina--Brady group $BB_{\Gamma}$ is defined to be the kernel of the group homomorphism $\phi\colon A_{\Gamma}\to\mathbb{Z}$, which sends all the generators to $1$. The main result of \cite{bestvinabradymoresetheoryandfinitenesspropertiesofgroups} states that $\Gamma$ is connected if and only if $BB_{\Gamma}$ is finitely generated; and the flag complex on $\Gamma$ is simply connected if and only if $BB_{\Gamma}$ is finitely presented. When a Bestvina--Brady group is finitely presented, it admits the following finite presentation, called the \emph{Dicks--Leary presentation}. 

\begin{thm}\textnormal{(\cite[Corollary 3]{dickslearypresentationsforsubgroupsofartingroups})}\label{Dicks--Leary presentations}
Let $\Gamma$ be a finite simplicial connected graph whose associated flag complex is simply connected. Then the group $BB_{\Gamma}$ is finitely presented, and it is generated by all the directed edges of $\Gamma$. The relators are of the form $e_{1}e_{2}=e_{3}=e_{2}e_{1}$, where $(e_{1},e_{2},e_{3})$ form a directed triangle in $\Gamma$; see Figure~\ref{Directed triangle}. Moreover, each directed edge $e=(v,w)$, as a generator of $BB_{\Gamma}$, embeds into $A_{\Gamma}$ as $e=vw^{-1}$.

\begin{figure}[H]
\begin{center}
\begin{tikzpicture}[scale=0.5]
\draw [thick] [middlearrow={stealth}] (4,3)--(0,0);
\draw [thick] [middlearrow={stealth}] (0,0)--(6,0);
\draw [thick] [middlearrow={stealth}] (4,3)--(6,0);

\draw [fill] (0,0) circle (4pt);
\draw [fill] (4,3) circle (4pt);
\draw [fill] (6,0) circle (4pt);

\node [left] at (2,2) {$e_{1}$};
\node at (3,-1) {$e_{2}$};
\node [right] at (5,2) {$e_{3}$};

\end{tikzpicture}
\end{center}
\caption{A directed triangle $(e_{1},e_{2},e_{3})$.}
\label{Directed triangle}
\end{figure}
\end{thm}

\begin{remark}
For the rest of the paper, every edge of a graph $\Gamma$, either undirected or directed, will be identified as an element of $BB_{\Gamma}$. We will ignore the orientation on $\Gamma$ if it is not being used in the argument. 
\end{remark}

The Dicks--Leary presentation in Theorem~\ref{Dicks--Leary presentations} is not necessary a minimal presentation. In fact, the generating set can be reduced to the set of the directed edges of a spanning tree of $\Gamma$, and all the relators are commutators; see \cite[Corollary 2.3]{papadimaandsuciualgebraicinvariantsforbestvinabradygroups}. 

\begin{example}\label{example of free group for D-I presentation}
Let $\Gamma$ be a tree on $n$ vertices, then $\Gamma$ has $n-1$ edges. Since $\Gamma$ has no triangles, the group $BB_\Gamma$ has no relators. Thus, we have $BB_{\Gamma}\cong F_{n-1}$, the free group of rank $n-1$.
\end{example}

\begin{example}\label{example of abelian group for D-I presentation}
Let $\Gamma$ be a clique $K_{n}$. Choose a vertex $v$ and a spanning tree $T$ for $\Gamma$ that consists of all the edges incident to $v$. Then it follows from the Dicks--Leary presentation that $BB_{\Gamma}$ is generated by the directed edges of $T$, and each generator commutes with all other generators. Since $T$ consists of $n-1$ edges, we have $BB_{\Gamma}\cong\mathbb{Z}^{n-1}$.
\end{example}

\begin{example}\label{abelian BBGs implies the defining graph is a clique}
Let $\Gamma$ be a finite simplicial graph and suppose $BB_{\Gamma}\cong\mathbb{Z}^{n-1}$. By definition, we have a split short exact sequence $1\to\mathbb{Z}^{n-1}\to A_{\Gamma}\to\mathbb{Z}\to1$. Equivalently, we have $A_{\Gamma}\cong\mathbb{Z}^{n-1}\oplus\mathbb{Z}\cong\mathbb{Z}^{n}$. Thus, the graph $\Gamma$ is the clique $K_{n}$.
\end{example}

\subsection{Groups acting on trees} 

Let $G$ be a group acting on a tree $T$. We assume that all the actions are without inversions. We say that the tree $T$ is \emph{non-trivial} if the action has no global fixed points. We say that an element $g\in G$ is \emph{elliptic}, or, $g$ acts on $T$ \emph{elliptically}, if $g$ fixes a point in $T$; an element $g\in G$ is \emph{hyperbolic} if it is not elliptic. Similarly, a subgroup $H<G$ is \emph{elliptic}, or, $H$ acts on $T$ \emph{elliptically}, if it fixes a point in $T$. We denote the fixed points set of $g$ and $H$ by $\mathrm{Fix}(g)$ and $\mathrm{Fix}(H)$, respectively. For an elliptic element $g\in G$, the set $\mathrm{Fix}(g)$ is a subtree of $T$. When $g\in G$ is hyperbolic, it preserves a line in $T$ on which it acts by translation. We call the invariant line an \emph{axis} of $g$ and denote it by $\mathrm{Axis}(g)$. The following two lemmas will be used frequently.

\begin{lemma}\textnormal{(\cite[Lemma 1.1]{cullervogtmannagrouptheoreticcriterionforpropertyFA})}\label{commutating element preserves axis}
Let $G$ be a group acting on a tree $T$. If $h\in G$ is hyperbolic and $g\in G$ commutes with $h$, then $\mathrm{Axis}(h)\subseteq\mathrm{Fix}(g)$. In particular, if $g$ and $h$ are commuting hyperbolic elements, then $\mathrm{Axis}(g)=\mathrm{Axis}(h)$.
\end{lemma}

\begin{lemma}\textnormal{(\cite[Lemma 1.1]{grovesandhullabeliansplittingsofraags}}\label{product of commutating elliptic elements is also elliptic}
Let $G$ be a group acting on a tree $T$. If $g_{1},g_{2}\in G$ are commuting elliptic elements, then $\mathrm{Fix}(g_{1})\cap\mathrm{Fix}(g_{2})\neq\phi$ and $g_{1}g_{2}$ is also an elliptic element.
\end{lemma}

The next two lemmas will be used in Section~\ref{section JSJ-decomposition of bestvina--brady groups}.

\begin{lemma}\label{generators of RAAG are elliptic implies generators of BBG are elliptic}
Let $\Gamma$ be a finite simplicial graph, and let $A_{\Gamma}$ act on a tree $T$. If all the generators of $A_{\Gamma}$ are elliptic, then all the generators of $BB_{\Gamma}$ are also elliptic.
\end{lemma}

\begin{proof}
Let $e=(v,w)\in E(\Gamma)$ be a directed edge. Since $v$ and $w$ are commuting elliptic elements, so are $v$ and $w^{-1}$. Therfore, the element $e=vw^{-1}$ is elliptic by Lemma~\ref{product of commutating elliptic elements is also elliptic}. 
\end{proof}

\begin{lemma}\label{H_K is elliptic for any clique K}
Let $\Gamma$ be a finite simplicial graph. If $BB_{\Gamma}$ acts on a tree $T$ such that each generator is elliptic, then $BB_{K}$ is elliptic for any clique $K$ in $\Gamma$.
\end{lemma}

\begin{proof}
Since $K$ is a clique, the group $BB_{K}$ is free abelian. Let $h_{1},\cdots,h_{n}$ be generators for $BB_{K}$. By Lemma~\ref{product of commutating elliptic elements is also elliptic}, the element $h_{i}h_{j}$ is elliptic since $h_{i}$ and $h_{j}$ are commuting elliptic elements for all $i$ and $j$. Thus, the group $BB_{K}$ is elliptic (\cite[p.64, Corollary 2]{serretrees}).
\end{proof}

We end this section by giving one more definition and a lemma. These will be used in the proof of Lemma~\ref{the star of a hyperbolic edge is clique}. Again, let $G$ be a group acting on a tree $T$. We now view $T$ as an $\mathbb{R}$-tree in the sense of \cite{CullerMorganGroupActionsonRTrees}. The \emph{translation length} $||g||$ of $g\in G$ is defined to be the infimum of the distance between $g$ and $gx$ over $x\in T$. If $||g||=0$, then $g$ is elliptic. If $||g||>0$, then $g$ is hyperbolic, and the action of $g$ on its axis translates points by $||g||$.

\begin{lemma}\label{higher power of commuting hyperbolic fix axis}
Let $g,h\in G$ be commuting hyperbolic elements. Then there are two nonzero integers $m$ and $n$ such that the element $g^mh^n$ fixes $\mathrm{Axis}(h)$ pointwisely.
\end{lemma}

\begin{proof}
Since $g$ and $h$ are commuting hyperbolic elements, for any nonzero integers $m$ and $n$, we have $\mathrm{Axis}(g^m)=\mathrm{Axis}(g)=\mathrm{Axis}(h)=\mathrm{Axis}(h^n)$ by Lemma~\ref{commutating element preserves axis}. If $g$ and $h$ move points in the opposite direction, then choose $m=||h||$ and $n=||g||$; if $g$ and $h$ move points in the same direction, then choose $m=||h||$ and $n=-||g||$. Then in both cases, the element $g^mh^n$ fixes $\mathrm{Axis}(h)$ pointwisely. 
\end{proof}

\section{Abelian Splittings of Finitely Presented Bestvina--Brady Groups}\label{section abelian splittings of bestvina--brady groups}

In this section we prove Theorem~\ref{abelian splittings of BBG}. Our proof relies on a series of lemmas. Given a finite simplicial graph $\Gamma$ and let $BB_{\Gamma}$ act on a tree $T$. We call $e\in E(\Gamma)$ a \emph{hyperbolic edge} (respectively \emph{elliptic edge}) if it acts hyperbolically (respectively elliptically) on $T$. The following definition is analogous to the links and stars of vertices.

\begin{definition}\label{definition of link and star of an edge}
Let $\Gamma$ be a finite simplicial graph, and let $e=(u,v)\in E(\Gamma)$. The \emph{link} of $e$, denoted by $\mathrm{link}(e)$, is defined to be the induced subgraph of $\Gamma$ on $\lbrace w\in V(\Gamma) \ \vert \ w \ \text{is adjacent to} \ \\ \text{both} \ u \ \text{and} \ v\rbrace$. The \emph{star} of $e$, denoted by $\mathrm{star}(e)$, is the induced subgraph on $\mathrm{link}(e)\cup\lbrace u,v\rbrace$.
\end{definition}

\begin{example}
Let $\Gamma$ be the graph as shown in Figure~\ref{example of the link of an edge}. Take $e=(a,c)$, then $\mathrm{link}(e)$ is $\lbrace b,d\rbrace$ and $\mathrm{star}(e)$ is the whole graph $\Gamma$. 

\begin{figure}[H]
\begin{tikzpicture}[scale=0.7]
\draw [thick] (-2,0)--(0,0)--(2,0);
\draw [thick] (0,2)--(-2,0);
\draw [thick] (0,2)--(0,0);
\draw [thick] (0,2)--(2,0);

\draw [fill] (0,2) circle (3pt);
\draw [fill] (-2,0) circle (3pt);
\draw [fill] (0,0) circle (3pt);
\draw [fill] (2,0) circle (3pt);

\node [above] at (0,2.1) {$a$};
\node [below left] at (-2,0) {$b$};
\node [below] at (0,-0.1) {$c$};
\node [below right] at (2,0) {$d$};
\end{tikzpicture}
\caption{}
\label{example of the link of an edge}
\end{figure}
\end{example}

We have the following observation.

\begin{lemma}\label{number of hyperbolic in a triangle}
Let $\Gamma$ be a finite simplicial connected graph such that the flag complex on $\Gamma$ is simply connected. Let $BB_{\Gamma}$ act on a tree $T$. 
\begin{enumerate}[ref=(\arabic*)]
\item \label{triangle in star(e) contains at most one elliptic edge} Any triangle in $\Gamma$ containing a hyperbolic edge has at most one elliptic edge. 
\item \label{having two elliptic edges implies three elliptic edges} If a triangle in $\Gamma$ contains two elliptic edges, then the third edge must be elliptic.
\end{enumerate}
\end{lemma}

\begin{proof} \
\begin{enumerate}
\item Let $e_{1},e_{2},e_{3}\in E(\Gamma)$ form a triangle. Suppose that $e_{1}$ is hyperbolic and $e_{2}$ and $e_{3}$ are elliptic. Given an appropriate orientation on $\Gamma$, we have $e_{2}e_{3}=e_{1}=e_{3}e_{2}$. Then it follows from Lemma~\ref{product of commutating elliptic elements is also elliptic} that $e_{1}$ is also elliptic, which is a contradiction.

\item This follows from \ref{triangle in star(e) contains at most one elliptic edge}.
\end{enumerate}
\end{proof}

\begin{lemma}\label{the star of a hyperbolic edge is clique}
Let $\Gamma$ be a finite simplicial connected graph such that the flag complex on $\Gamma$ is simply connected. Suppose that $BB_{\Gamma}$ acts on a tree $T$ with abelian edge stabilizers. If $e\in E(\Gamma)$ is hyperbolic, then $\mathrm{star}(e)$ is a clique.
\end{lemma}

\begin{proof} 
Denote $e=(u,v)$, and let $w_{1}$ and $w_{2}$ be any two vertices in $\mathrm{link}(e)$. By definition, the subsets $\lbrace w_1, u, v\rbrace$ and $\lbrace w_2, u, v\rbrace$ form two triangles in $\Gamma$. Each of these triangles contains at least two hyperbolic edges by Lemma~\ref{number of hyperbolic in a triangle} \ref{triangle in star(e) contains at most one elliptic edge}. Since $e$ is hyperbolic, let $e_1=(u,w_1)$ and $e_2=(v,w_2)$ also be hyperbolic (the argument will be similar for the case when $e_2=(u,w_2)$ is hyperbolic.)  Note that $e$ commutes with $e_1$ and $e_2$ in $BB_\Gamma$ with a suitable orientation on $\Gamma$. By Lemma~\ref{higher power of commuting hyperbolic fix axis}, there are nonzero integers $m_1$, $m_2$, $n_1$, and $n_2$ such that $e^{m_1}_{1}e^{n_1}$ and $e^{m_2}_{2}e^{n_2}$ fix $\mathrm{Axis}(e)$ pointwisely. Then the subgroup generalized by $e^{m_1}_{1}e^{n_1}$ and $e^{m_2}_{2}e^{n_2}$ fixes an edge of $\mathrm{Axis}(e)\subset T$, and therefore, it is an abelian group by the assumption. Moreover, since $e$ commutes with $e_1$ and $e_2$, we have $e^{m_1}_{1}e^{m_2}_{2}=e^{m_2}_{2}e^{m_1}_{1}$. Now consider the corresponding elements $e_1=uw^{-1}_1$ and $e_2=vw^{-1}_{2}$ in $A_\Gamma$. Since $u$ and $v$ commute with each other and both of them commute with $w_1$ and $w_2$ in $A_\Gamma$, the relation $e^{m_1}_{1}e^{m_2}_{2}=e^{m_2}_{2}e^{m_1}_{1}$ implies $w^{m_1}_{1}w^{m_2}_{2}=w^{m_2}_{2}w^{m_1}_{1}$, which holds only when $w_1$ and $w_2$ commute in $A_\Gamma$. That is, the vertices $w_{1}$ and $w_{2}$ are adjacent. Hence, the subgraph $\mathrm{link}(e)$ is a clique, and so is $\mathrm{star}(e)$.
\end{proof}

\begin{lemma}\label{observations of star(e) for hyperbolic edge e}
Assume that $\Gamma$ and $BB_{\Gamma}$ satisfy the assumptions in Lemma~\ref{the star of a hyperbolic edge is clique}, and let $e\in E(\Gamma)$ be hyperbolic. Then
\begin{enumerate}[ref=(\arabic*)]
\item \label{star(e')=star(e) if e' is hyperbolic} If an edge $e'$ in $\mathrm{star}(e)$ is hyperbolic, then $\mathrm{star}(e')=\mathrm{star}(e)$.
\item \label{no paths in Gamma-star(e) connect the end points of a hyperbolic edge in star(e)} If an edge $e'=(v',w')$ in $\mathrm{star}(e)$ is hyperbolic, then there does not exist a path in $\Gamma\setminus\mathrm{star}(e)$ between $v'$ and $w'$.
\end{enumerate}
\end{lemma}

\begin{proof} \
\begin{enumerate}
\item If $e'=e$, then the statement is obvious. Suppose $e'\neq e$. Since both $e$ and $e'$ are hyperbolic, the subgraphs $\mathrm{star}(e)$ and $\mathrm{star}(e')$ are cliques by Lemma~\ref{the star of a hyperbolic edge is clique}. Denote $e=(v,w)$ and $e'=(v',w')$. We have two cases. The first case is that the edges $e$ and $e'$ are incident, say $w=w'$. Then $v$ is in $\mathrm{star}(e')$ by Definition~\ref{definition of link and star of an edge}. For any vertex $u$ in $\mathrm{star}(e')$, the vertices $u$ and $v$ are adjacent since $\mathrm{star}(e')$ is a clique. Thus, the vertex $u$ is in $\mathrm{star}(e)$. Therefore, we have $V(\mathrm{star}(e'))\subseteq V(\mathrm{star}(e))$. Switching the roles of $e$ and $e'$ gives $V(\mathrm{star}(e))\subseteq V(\mathrm{star}(e'))$. Thus, we have $V(\mathrm{star}(e'))=V(\mathrm{star}(e))$. Hence, we obtain $\mathrm{star}(e')=\mathrm{star}(e)$. This proves the first case. The second case is that $e$ and $e'$ are not incident; see Figure~\ref{the case of e and e' are not incident}. Since $\mathrm{star}(e)$ is a clique, the vertices $v$ and $w$ are adjacent to both $v'$ and $w'$. By Lemma~\ref{number of hyperbolic in a triangle} \ref{triangle in star(e) contains at most one elliptic edge}, at least one of the edges $(v,v')$ and $(w,v')$ is hyperbolic, say $e''=(w,v')$ is hyperbolic. Now, we have $e$, $e'$, and $e''$ are hyperbolic, and $e''$ is incident to both $e$ and $e'$. Then it follows from the first case that $\mathrm{star}(e')=\mathrm{star}(e'')=\mathrm{star}(e)$. This proves the second case. 

\begin{figure}[H]
\begin{center}
\begin{tikzpicture}[scale=0.7]

\draw [thick,red] (0,2)--(0,-0.5);
\draw [thick,red] (0,-0.5)--(-2,-2);
\draw [thick] (0,-0.5)--(2,-2);
\draw [thick] (-2,-2)--(0,2);
\draw [thick,red] (0,2)--(2,-2);
\draw [thick,red] (-2,-2)--(2,-2);

\draw [fill] (0,2) circle (3pt);
\draw [fill] (0,-0.5) circle (3pt);
\draw [fill] (-2,-2) circle (3pt);
\draw [fill] (2,-2) circle (3pt);

\node [above] at (0,2) {\small $v$};
\node [below] at (0,-0.5) {\small $w$};
\node [left] at (-2,-2) {\small $v'$};
\node [right] at (2,-2) {\small $w'$};

\node [right] at (-0.1,0.5) {\small $e$};
\node [below] at (0,-2) {\small $e'$};
\node [below] at (-0.7,-0.85) {\small $e''$};
\end{tikzpicture}
\end{center}
\caption{The red edges are hyperbolic.}
\label{the case of e and e' are not incident}
\end{figure}

\item Recall that $\mathrm{star}(e')=\mathrm{star}(e)$ is a clique. Suppose that there is a shortest path $p=(v',v'_{1},\cdots,v'_{k},w')$ from $v'$ to $w'$, where $v'_{i}$ are vertices in $\Gamma\setminus\mathrm{star}(e)$ for $1\leq i\leq k$. Note that none of the vertices $v'_{2},\cdots,v'_{k-1}$ is adjacent to either $v'$ or $w'$. Otherwise, the path $p$ would not be the shortest. Also, the vertices $v'_{1}$ and $v'_{k}$ cannot be adjacent to $w'$ and $v'$, respectively. Otherwise, they would belong to $\mathrm{star}(e')=\mathrm{star}(e)$. Let $u'\in\mathrm{star}(e)$ be a vertex. Then at least one of the edges $(v',u')$ and $(w',u')$ is hyperbolic by Lemma~\ref{number of hyperbolic in a triangle} \ref{triangle in star(e) contains at most one elliptic edge}, say $(v',u')$. If a vertex $v'_{i}$ is adjacent to $u'$ for some $i\in\lbrace2,\cdots,k-1\rbrace$, then $(v',v'_{1},\cdots,v'_{i},u')$ is a path in $\Gamma\setminus\mathrm{star}(e)$ between the two end vertices of the hyperbolic edge $(v',u')$; see Figure~\ref{a path connecting the endpoints of a hyperbolic edge}. Thus, without loss of generality, we assume that none of the vertices $v'_{2},\cdots,v'_{k-1}$ is adjacent to vertices in $\mathrm{star}(e)$. Since the flag complex on $\Gamma$ is simply connected, the cycle $(v',v'_{1},\cdots,v'_{k},w',v')$ is the boundary $\partial D$ of a triangulated disk $D$ in $\left(\Gamma\setminus\mathrm{star}(e)\right)\cup\lbrace e'\rbrace$. 

\begin{figure}[H]
\begin{center}
\begin{tikzpicture}[scale=0.6]
\draw [thick] (0,0)--(2,0);
\node [below] at (1,0) {$e'$};

\draw [thick, orange] (0,0)--(-1,1)--(0,3)--(2,4)--(1,1);
\draw [thick] (2,4)--(4,3)--(5,2)--(4,1)--(2,0);

\draw [thick, red] (1,1)--(0,0);
\draw [thick] (1,1)--(2,0);
\draw [fill] (1,1) circle (3pt);
\node [right] at (1,1.1) {\small $u'$};

\draw [fill] (0,0) circle (3pt);
\node [below] at (0,0) {\small $v'$};
\draw [fill] (-1,1) circle (3pt);
\node [left] at (-1,1) {\small $v'_{1}$};
\draw [fill] (0,3) circle (3pt);
\draw [fill] (2,4) circle (3pt);
\node [above] at (2,4) {\small $v'_{i}$};
\draw [fill] (4,3) circle (3pt);
\draw [fill] (5,2) circle (3pt);
\draw [fill] (4,1) circle (3pt);
\node [below right] at (4,1) {\small $v'_{k}$};
\draw [fill] (2,0) circle (3pt);
\node [below] at (2,0) {\small $w'$};
\end{tikzpicture}
\end{center}
\caption{The orange path is in $\Gamma\setminus\mathrm{star}(e)$, connecting two endpoints of a hyperbolic edge (red) in $\mathrm{star}(e')$.}
\label{a path connecting the endpoints of a hyperbolic edge}
\end{figure}

We claim that such a disk $D$ does not exist. Denote the interior of $D$ by $\interior(D)$. Since $D$ is a triangulated disk and the vertices on $\partial D$ are not adjacent to the vertices of $\mathrm{star}(e)$, except for the vertices $v'_1$ and $v'_k$, there is a vertex $u$ in $\interior(D)$ that is adjacent to both $v'$ and $w'$. This implies that $u\in\mathrm{star}(e)$. Thus, there are some vertices, and hence, edges of $\mathrm{star}(e)$, lying in $\interior(D)$. Next, we show that this cannot happen.

Since $\mathrm{star}(e)$ is a clique, there are at most two vertices from $\mathrm{star}(e)$ lying in $\interior(D)$. Otherwise, the disk $D$ would contain a $K_{5}$, which is impossible since $K_{5}$ is not planar. Let $w$ be the only vertex in $\mathrm{star}(e)$ lying in $\interior(D)$. Note that $w$, $v'$, and $w'$ form a triangle since $\mathrm{star}(e)$ is a clique. Recall that the vertex $w$ is not adjacent to any vertex on $\partial D$ other than $v'$ and $w'$. Since $D$ is a triangulated disk, there are vertices $v_1$ and $v_2$ in $\interior(D)$ such that $v_1$, $w$, $v'$ and $v_2$, $w$, $w'$ form triangles; see Figure~\ref{the interior of D cannot contain only one vertex in star(e)}. By Lemma~\ref{number of hyperbolic in a triangle} \ref{triangle in star(e) contains at most one elliptic edge}, either $(w,v')$ or $(w,w')$ is hyperbolic. Say, the edge $(w,v')$ is hyperbolic. Then we have $v_1\in\mathrm{star}((w,v'))=\mathrm{star}(e)$ by \ref{star(e')=star(e) if e' is hyperbolic}, which contradicts the uniqueness of $w$.

\begin{figure}[H]
\begin{tikzpicture}[scale=0.635]
\draw [thick, red] (4,0)--(8,0);

\draw [thick, red] (6,4)--(4,0);
\draw [thick] (6,4)--(8,0);
\node [above] at (6,4) {$w$};
\node [below] at (6,0) {$e'$};
\node [below] at (4,0) {$v'$};
\node [below] at (8,0) {$w'$};

\draw [thick] (4,0)--(2,1)--(1,3)--(2,5)--(4,6)--(8,6)--(10,5)--(11,3)--(10,1)--(8,0);

\draw [fill] (2,1) circle (3pt);
\draw [fill] (1,3) circle (3pt);
\draw [fill] (2,5) circle (3pt);
\draw [fill] (4,6) circle (3pt);
\draw [fill] (6,6) circle (3pt);
\draw [fill] (8,6) circle (3pt);
\draw [fill] (10,5) circle (3pt);
\draw [fill] (11,3) circle (3pt);
\draw [fill] (10,1) circle (3pt);

\node [left] at (2,1) {$v'_1$};
\node [left] at (1,3) {$v'_2$};
\node [above] at (6,6) {$v'_i$};
\node [right] at (11,3) {$v'_{k-1}$};
\node [right] at (10,1) {$v'_k$};

\draw [thick, dashed] (3,2)--(8,0);

\draw [thick] (4,0)--(3,2)--(6,4);
\draw [fill, green] (3,2) circle (3pt);
\node [left] at (3,2) {$v_1$};
\draw [thick] (8,0)--(9,2)--(6,4);
\draw [fill] (9,2) circle (3pt);
\node [right] at (9,2) {$v_2$};

\draw [fill, green] (4,0) circle (3pt);
\draw [fill, green] (8,0) circle (3pt);
\draw [fill, green] (6,4) circle (3pt);

\end{tikzpicture}
\caption{This picture illustrates a triangulated disk $D$ containing a vertex $w$ in $\mathrm{star}(e)$ must contain another vertex $v_1$ in $\mathrm{star}(e)$. The red edges are hyperbolic, and the green vertices are in $\mathrm{star}(e)$.}
\label{the interior of D cannot contain only one vertex in star(e)}
\end{figure}

Now, let $v$ and $w$ be two vertices of $\mathrm{star}(e)$ lying in $\interior(D)$. Again, the vertices $v$ and $w$ together with $v'$ and $w'$ form a $K_{4}$ since $\mathrm{star}(e)$ is a clique. The same argument in the previous paragraph shows that there is a vertex $u_{1}$ in $\interior(D)$, different from $v$ and $w$, belonging to $\mathrm{star}(e)$; see Figure~\ref{a K4 of the star(e') is in the interior of D}. Therefore, the vertex $u_{1}$ together with the existing $K_{4}$ form a $K_{5}$ in $D$, which is absurd. Thus, we have showed that $\interior(D)$ contains no vertices and edges from $\mathrm{star}(e)$. Hence, the disk $D$ does not exist. This proves the claim.

Since the disk $D$ does not exist, its boundary $\partial D$ does not exist either. Hence, there is no path $p=(v',v'_{1},\cdots,v'_{k},w')$ in $\Gamma\setminus\mathrm{star}(e)$ from $v'$ to $w'$, where $e'=(v',w')\in\mathrm{star}(e)$ is an hyperbolic edge.

\begin{figure}[H]
\begin{tikzpicture}[scale=0.635]
\draw [thick, red] (4,0)--(8,0);
\draw [thick] (8,0)--(6,4);
\draw [thick, red] (4,0)--(6,4);
\draw [thick] (6,1.5)--(6,4);
\node [below] at (6,1.5) {$v$};
\node [above] at (6,4) {$w$};
\draw [thick] (6,1.5)--(4,0);
\draw [thick] (6,1.5)--(8,0);
\node [below] at (6,0) {$e'$};
\node [below] at (4,0) {$v'$};
\node [below] at (8,0) {$w'$};

\draw [thick] (4,0)--(2,1)--(1,3)--(2,5)--(4,6)--(8,6)--(10,5)--(11,3)--(10,1)--(8,0);

\draw [fill] (2,1) circle (3pt);
\draw [fill] (1,3) circle (3pt);
\draw [fill] (2,5) circle (3pt);
\draw [fill] (4,6) circle (3pt);
\draw [fill] (6,6) circle (3pt);
\draw [fill] (8,6) circle (3pt);
\draw [fill] (10,5) circle (3pt);
\draw [fill] (11,3) circle (3pt);
\draw [fill] (10,1) circle (3pt);

\node [left] at (2,1) {$v'_1$};
\node [left] at (1,3) {$v'_2$};
\node [above] at (6,6) {$v'_i$};
\node [right] at (11,3) {$v'_{k-1}$};
\node [right] at (10,1) {$v'_k$};

\draw [thick, dashed] (3,2)--(6,1.5);
\draw [thick, dashed] (3,2)--(8,0);

\draw [thick] (4,0)--(3,2)--(6,4);
\draw [fill, green] (3,2) circle (3pt);
\node [left] at (3,2) {$u_1$};

\draw [fill, green] (4,0) circle (3pt);
\draw [fill, green] (8,0) circle (3pt);
\draw [fill, green] (6,1.5) circle (3pt);
\draw [fill, green] (6,4) circle (3pt);

\end{tikzpicture}
\caption{This picture illustrates a triangulated disk $D$ containing a $K_4$ in $\mathrm{star}(e)$ must contain a $K_5$, which is impossible. The red edges are hyperbolic, and the green vertices are in $\mathrm{star}(e)$.}
\label{a K4 of the star(e') is in the interior of D}
\end{figure}

\end{enumerate}
\end{proof}

\begin{definition}
Let $\Gamma'$ be a connected subgraph of a graph $\Gamma$. We call $\Gamma'$ an \emph{elliptic component} of $\Gamma$ if $E(\Gamma')$ contains only elliptic edges. 
\end{definition}

\begin{lemma}\label{star(e) contains a separating clique}
Assume that $\Gamma$ and $BB_{\Gamma}$ satisfy the assumptions in Lemma~\ref{the star of a hyperbolic edge is clique}, and let $e\in E(\Gamma)$ be hyperbolic. Suppose that $\Gamma$ has no cut-vertices and $\Gamma\setminus\mathrm{star}(e)$ is nonempty. Then $\Gamma$ contains a separating clique. 
\end{lemma}

\begin{proof}
If all the edges of $\mathrm{star}(e)$ are hyperbolic, then it follows from Lemma~\ref{observations of star(e) for hyperbolic edge e} \ref{no paths in Gamma-star(e) connect the end points of a hyperbolic edge in star(e)} that there are no paths in $\Gamma\setminus\mathrm{star}(e)$ between any two distinct vertices in $\mathrm{star}(e)$. This implies that the graph $\Gamma$ is disconnected, which is a contradiction. Thus, the subgraph $\mathrm{star}(e)$ must contain an elliptic edge. 

Suppose that $\mathrm{star}(e)$ contains only one elliptic component $\Gamma'$. Observe that $\Gamma'$ is a clique by \ref{triangle in star(e) contains at most one elliptic edge} and \ref{having two elliptic edges implies three elliptic edges} of Lemma~\ref{number of hyperbolic in a triangle}. Let $u\in V(\Gamma\setminus\mathrm{star}(e))$. Since $\Gamma$ is connected and has no cut-vertices, there is a path in $\Gamma\setminus\mathrm{star}(e)$ containing $u$ that connects two vertices $v'$ and $w'$ in $\mathrm{star}(e)$. By Lemma~\ref{observations of star(e) for hyperbolic edge e} \ref{no paths in Gamma-star(e) connect the end points of a hyperbolic edge in star(e)}, the vertices $v'$ and $w'$ belong to $V(\Gamma')$. Thus, the graph $\Gamma\setminus\Gamma'$ is disconnected.

Now suppose that $\mathrm{star}(e)$ contains two elliptic components $\Gamma_{1}$ and $\Gamma_{2}$. Recall that $\Gamma_{1}$ and $\Gamma_{2}$ are cliques. If $V(\Gamma_{1})\cap V(\Gamma_{2})$ is nonempty, then the subgraph $\Gamma''$ induced by $V(\Gamma_{1})\cup V(\Gamma_{2})$ is an elliptic component in $\mathrm{star}(e)$. Therefore, the graph $\Gamma''$ is a separating clique of $\Gamma$ by the previous paragraph. So we assume that $V(\Gamma_{1})\cap V(\Gamma_{2})$ is empty. Note that all the edges between $\Gamma_{1}$ and $\Gamma_{2}$ must be hyperbolic. Let $v_{1}\in V(\Gamma_{1})$, $v_{2}\in V(\Gamma_{2})$, and $u\in V(\Gamma\setminus\mathrm{star}(e))$. Since $(v_{1},v_{2})$ is hyperbolic, by Lemma~\ref{observations of star(e) for hyperbolic edge e} \ref{no paths in Gamma-star(e) connect the end points of a hyperbolic edge in star(e)}, there does not exist a path between $v_{1}$ and $v_{2}$ in $\Gamma\setminus\mathrm{star}(e)$ that contains $u$. That is, any path in $\Gamma\setminus\mathrm{star}(e)$ containing $u$ must start and end at the same elliptic component. Thus, either $\Gamma_{1}$ or $\Gamma_{2}$ is a separating clique of $\Gamma$.

If $\mathrm{star}(e)$ contains more than two disjoint elliptic components, then by the previous paragraph, each of these elliptic components is a separating clique of $\Gamma$. This completes the proof.
\end{proof}

To make a comparison to our main result, we state Groves and Hull's result for RAAGs:

\begin{thm}\textnormal{(\cite[Theorem A]{grovesandhullabeliansplittingsofraags})}\label{abelian splittings of raags}
Let $\Gamma$ be a finite simplicial graph. The right-angled Artin group $A_{\Gamma}$ splits over an abelian subgroup if and only if one of the following occurs:
\begin{enumerate}
\item $\Gamma$ is disconnected;
\item $\Gamma$ is a complete graph;
\item $\Gamma$ contains a separating clique.
\end{enumerate}
\end{thm}

Now we prove our main theorem.

\begin{thm}\label{abelian splittings of BBG}
Let $\Gamma$ be a finite simplicial connected graph such that the flag complex on $\Gamma$ is simply connected. The Bestvina--Brady group $BB_{\Gamma}$ splits over an abelian subgroup if and only if $\Gamma$ satisfies one of the following:
\begin{enumerate}[ref=(\arabic*)]
\item \label{has a cut-vertex} $\Gamma$ has a cut-vertex;
\item \label{is a complete graph} $\Gamma$ is a complete graph;
\item \label{has a separating clique} $\Gamma$ has a separating clique $K_{n}$, $n\geq2$ .
\end{enumerate}
\end{thm}

\begin{proof}
One direction is obvious. That is, if one of \ref{has a cut-vertex}, \ref{is a complete graph}, or \ref{has a separating clique} occurs, then $BB_{\Gamma}$ splits over an abelian subgroup.

Suppose that $\Gamma$ is not a complete graph and has no cut-vertices, and let $BB_{\Gamma}$ act non-trivially on a tree $T$ with abelian edge stabilizers. If $\Gamma$ contains a hyperbolic edge, then it contains a separating clique by Lemma~\ref{star(e) contains a separating clique}. 

Now suppose that all the edges of $\Gamma$ act elliptically on $T$. We define a map $F:\Gamma\rightarrow T$ as follows. Since the action of $BB_{\Gamma}$ on $T$ has no global fixed points, the intersection ${\displaystyle\cap_{e\in E(\Gamma)}\mathrm{Fix}(e)}$ is empty. Thus, there exist $e_{\alpha},e_{\beta}\in E(\Gamma)$ such that $\mathrm{Fix}(e_{\alpha})$ and $\mathrm{Fix}(e_{\beta})$ are disjoint. Let $L$ be the geodesic between $\mathrm{Fix}(e_{\alpha})$ and $\mathrm{Fix}(e_{\beta})$ and pick a point $p$ in the interior of an edge of $L$. Choose $y_{\alpha}\in\mathrm{Fix}(e_{\alpha})$, $y_{\beta}\in\mathrm{Fix}(e_{\beta})$ and define $F(e_{\alpha})=y_{\alpha}$, $F(e_{\beta})=y_{\beta}$, and $F(e_{s})=p$ whenever $p\in\mathrm{Fix}(e_{s})$ for $e_{s}\in E(\Gamma)$. For other $e_{t}\in E(\Gamma)$, choose $y_{t}\in\mathrm{Fix}(e_{t})$ and define $F(e_{t})=y_{t}$. Next, we describe how the map $F$ sends vertices of $\Gamma$ to $T$. For each $v\in V(\Gamma)$, denote $e_{1},\cdots,e_{k}$ to be the edges incident to $v$. Let $L_{ij}$ be the geodesic between $F(e_{i})$ and $F(e_{j})$ for commuting $e_{i}$ and $e_{j}$. Define
$$
F(v)=\bigcup_{i,j}L_{ij}.
$$
Since $\Gamma$ has no cut-vertices and the flag complex on $\Gamma$ is simply connected, each edge of $\Gamma$ belongs to a triangle, and $F(v)$ and $\bigcup_{v\in V(\Gamma)}F(v)$ are connected. When $e_{i}$ and $e_{j}$ commute, the intersection $\mathrm{Fix}(e_{i})\cap\mathrm{Fix}(e_{j})$ is nonempty by Lemma~\ref{product of commutating elliptic elements is also elliptic}. Thus, the geodesic $L_{ij}$ lies entirely in $\mathrm{Fix}(e_{i})\cup\mathrm{Fix}(e_{j})$. Notice that $L_{ij}$ can be degenerate. Hence, we have $F(\Gamma)\subseteq\bigcup_{e\in E(\Gamma)}\mathrm{Fix}(e)$. Moreover, the map $F$ sends connected sets to connected sets. Indeed, let $e=(v,w)$ be an edge of $\Gamma$. Then $F(v)\cap F(w)$ is nonempty since it contains the point $\mathrm{F}(e)$. Therefore, the set $F(v)\cup F(w)$ is connected whenever $v$ and $w$ are adjacent vertices.

Since the set $F^{-1}(p)$ contains all the edges in $\Gamma$ that fix $p$, it fixes the edge containing $p$. Then the group generated by $F^{-1}(p)$ also fixes the edge containing $p$, and therefore, it is an abelian group by the assumption. Hence, the set $F^{-1}(p)$ forms a clique $K$ in $\Gamma$. Since $p$ separates $T$ and the map $F$ sends connected sets to connected sets, the clique $K$ separates~$\Gamma$.
\end{proof}

\begin{remark}\label{remark on distinguishing cut-vertices and cliques}
The cut-vertices and separating cliques of sizes greater than two are distinguished in Theorem~\ref{abelian splittings of BBG}, but not in Theorem~\ref{abelian splittings of raags}. This is because when $\Gamma$ has a cut-vertex, the group $A_{\Gamma}$ splits as an amalgamated product over $\mathbb{Z}$ while $BB_{\Gamma}$ splits as a free product. 
\end{remark}

Corollary~\ref{RAAGs split iff BBGs split} is an immediate consequence of Theorem~\ref{abelian splittings of raags} and Theorem~\ref{abelian splittings of BBG}.

\begin{cor}\label{RAAGs split iff BBGs split}
Let $\Gamma$ be a finite simplicial connected graph whose associated flag complex is simply connected. Suppose that $\Gamma$ has no cut-vertices. Then $A_{\Gamma}$ splits over an abelian subgroup if and only if $BB_{\Gamma}$ splits over an abelian subgroup.
\end{cor}

If we see a separating clique of a particular size in the defining graph, then the associated RAAG and Bestvina--Brady group split as we expect.

\begin{cor}\label{Gamma contains a separating n-clique, then RAAG and BBG split over Z^n and Z^n-1, respectively}
Let $\Gamma$ be a finite simplicial connected graph whose associated flag complex is simply connected. If $\Gamma$ contains a separating clique $K_{n}$, $n\geq2$, then $A_{\Gamma}$ splits over $\mathbb{Z}^{n}$ and $BB_{\Gamma}$ splits over $\mathbb{Z}^{n-1}$.
\end{cor}

\begin{remark}
Zaremsky \cite{zaremskycommensurabilityinvarianceforabeliansplittingsofrightangleartingroupsbraidgroupsandloopbraidgroups} pointed out that when a RAAG splits over an abelian subgroup, the size of the separating clique in the defining graph is unknown. He proved that for a non-complete finite simplicial graph $\Gamma$, if $A_{\Gamma}$ splits over $\mathbb{Z}^{n}$, then $\Gamma$ admits a separating clique of size $k$ for some $k\leq n$; see \cite[Proposition 2.3]{zaremskycommensurabilityinvarianceforabeliansplittingsofrightangleartingroupsbraidgroupsandloopbraidgroups}. Our proof does not detect the sizes of the separating cliques for splitting Bestvina--Brady groups either. However, if one can control the size of the separating clique in either case, then the other case follows.
\end{remark}

\section{JSJ-decompositions of Finitely Presented Bestvina--Brady Groups}\label{section JSJ-decomposition of bestvina--brady groups}

In this section, we describe the JSJ-decompositions for finitely presented Bestvina--Brady groups which split over abelian subgroups. For completeness, we give necessary definitions along the way. Throughout this section, the graph $\Gamma$ will always be a finite simplicial connected graph whose associated flag complex is simply connected.

In \cite{grovesandhullabeliansplittingsofraags}, the authors gave a graph of groups decomposition of $A_{\Gamma}$ over abelian subgroups such that each vertex of $\Gamma$ acts on the Bass--Serre tree elliptically. They called such a decomposition a \emph{vertex-elliptic abelian splitting}. When $A_{\Gamma}$ admits a vertex-elliptic abelian splitting, by Corollary~\ref{RAAGs split iff BBGs split} and Lemma~\ref{generators of RAAG are elliptic implies generators of BBG are elliptic}, the group $BB_{\Gamma}$ also has a graph of groups decomposition over abelian subgroups such that each directed edge of $\Gamma$ acts on the Bass--Serre tree elliptically. We call such a splitting an \emph{edge-elliptic abelian splitting}. We will describe these splittings later with more details. 

We recall some terminologies from \cite{guirardelandlevittjsjdecompositionsofgroups}. Fix a family $\mathcal{A}$ of subgroups of a group $G$ that is stable under conjugate and taking subgroups. If $G$ acts on a tree $T$ such that all the edge stabilizers are in $\mathcal{A}$, then we call $T$ an \emph{$\mathcal{A}$-tree}. An $\mathcal{A}$-tree is called \emph{universally elliptic} if its edge stabilizers fix a point in every $\mathcal{A}$-tree. A tree $T$ \emph{dominates} another tree $T'$ if every subgroup of $G$ that fixes a point in $T$ also fixes a point in $T'$.

\begin{definition}\label{definition of JSJ trees}
A \emph{JSJ-tree} of $G$ over $\mathcal{A}$ is an $\mathcal{A}$-tree $T$ such that it is universally elliptic and dominates any other universally elliptic $\mathcal{A}$-tree. The graph of groups decomposition $\mathcal{G}=T/G$ is called a \emph{JSJ-decomposition of $G$ over $\mathcal{A}$}.
\end{definition}

If we further fix another family $\mathcal{H}$ of subgroups of $G$ and require that each $H\in\mathcal{H}$ fixes a point in an $\mathcal{A}$-tree $T$, then $T$ is called an \emph{$(\mathcal{A},\mathcal{H})$-tree}. If an $(\mathcal{A},\mathcal{H})$-tree $T$ is universally elliptic and dominates any other universally elliptic $(\mathcal{A},\mathcal{H})$-tree, then $T$ is called a \emph{JSJ-tree over $\mathcal{A}$ relative to $\mathcal{H}$}, and the corresponding graph of groups decomposition $\mathcal{G}=T/G$ is called a \emph{JSJ-decomposition of $G$ over $\mathcal{A}$ relative to $\mathcal{H}$}.

\begin{definition}
An \emph{edge-elliptic abelian JSJ-decomposition} of $BB_{\Gamma}$ is a JSJ-decomposition over $\mathcal{A}$ relative to $\mathcal{H}$, where $\mathcal{A}$ is the family of abelian subgroups of $BB_{\Gamma}$ and $\mathcal{H}=\big\lbrace\langle e\rangle \ \vert \ e\in E(\Gamma)\big\rbrace$. 
\end{definition}

We now describe a graph of groups decomposition of $A_{\Gamma}$ given in \cite{grovesandhullabeliansplittingsofraags}. For our purpose, we further assume that $\Gamma$ has no cut-vertices. Suppose that $A_{\Gamma}$ acts on a tree such that each generator acts elliptically. Observe that $\Gamma$ cannot be a complete graph; otherwise, the action would be trivial. Let $K_{1},\cdots,K_{n}$ be the minimal size separating $k$-cliques of $\Gamma$, $k\geq 2$. By Theorem~\ref{abelian splittings of raags}, each of these separating $k$-cliques gives a splitting of $A_{\Gamma}$ over $A_{K_{i}}$, and the vertex groups are $A_{\Gamma_{\alpha\beta}\cup K_{i}}$, where $\Gamma_{\alpha\beta}$ is a connected component of $\Gamma\setminus K_{i}$. For each vertex group $A_{\Gamma_{\alpha\beta}\cup K_{i}}$, spot the minimal size (greater than $k$) separating cliques of $\Gamma_{\alpha\beta}\cup K_{i}$. Then each of these separating cliques gives a splitting of the vertex group $A_{\Gamma_{\alpha\beta}\cup K_{i}}$ over abelian subgroups defined by the associated RAAGs of those separating cliques. Continue this procedure, we obtain a graph of groups decomposition $\mathcal{G}_{A_{\Gamma}}$ of $A_{\Gamma}$. In fact, the decomposition $\mathcal{G}_{A_{\Gamma}}$ is a vertex-elliptic abelian splitting. We refer the reader to \cite{grovesandhullabeliansplittingsofraags} for a more detailed description and the proof of the following theorem.

\begin{thm}\textnormal{(\cite[Theorem 2.4]{grovesandhullabeliansplittingsofraags})}
The graph of groups decomposition $\mathcal{G}_{A_{\Gamma}}$ described above is a vertex-elliptic abelian JSJ-decomposition of $A_{\Gamma}$.
\end{thm}

For a vertex-elliptic abelian splitting $\mathcal{G}_{A_{\Gamma}}$, replace all the vertex groups and edge groups with their corresponding Bestvina--Brady groups. By Theorem~\ref{abelian splittings of BBG} and Corollary~\ref{RAAGs split iff BBGs split}, we obtain a graph of groups decomposition $\mathcal{G}_{BB_{\Gamma}}$ of $BB_{\Gamma}$. Moreover, it follows from Lemma~\ref{generators of RAAG are elliptic implies generators of BBG are elliptic} and Corollary~\ref{Gamma contains a separating n-clique, then RAAG and BBG split over Z^n and Z^n-1, respectively} that $\mathcal{G}_{BB_{\Gamma}}$ is an edge-elliptic abelian splitting. Note that from our construction, the underlying graphs of the graph of groups decompositions $\mathcal{G}_{A_{\Gamma}}$ and $\mathcal{G}_{BB_{\Gamma}}$ are the same. It was shown in \cite[Proposition 2.1]{grovesandhullabeliansplittingsofraags} that the underlying graph of a vertex-elliptic abelian splitting $\mathcal{G}_{A_{\Gamma}}$ is a tree. Thus, we have the following proposition.

\begin{prop}\label{the underlying graph of an edge-elliptic splitting is a tree}
The underlying graph of an edge-elliptic abelian splitting $\mathcal{G}_{BB_{\Gamma}}$ is a tree, where $\Gamma$ is not a complete graph.
\end{prop}

We are ready to give a JSJ-decomposition of $BB_{\Gamma}$.

\begin{thm}\label{edge-elliptic abelian JSJ-decomposition of BBGs}
Suppose that $\Gamma$ is not complete. The graph of groups decomposition $\mathcal{G}_{BB_{\Gamma}}$ of $BB_{\Gamma}$ is an edge-elliptic abelian JSJ-decomposition.
\end{thm}

\begin{proof}
Denote $T_{A}$ and $T_{BB}$ the Bass--Serre tree of $\mathcal{G}_{A_{\Gamma}}$ and $\mathcal{G}_{BB_{\Gamma}}$, respectively. Since $BB_{\Gamma}$ is a subgroup of $A_{\Gamma}$, the graph $T_{BB}$ is a subtree of $T_{A}$, and each generator of $A_{\Gamma}$ acts on $T_{BB}$ elliptically. By Lemma~\ref{generators of RAAG are elliptic implies generators of BBG are elliptic},  each generator of $BB_{\Gamma}$ also acts on $T_{BB}$ elliptically. Thus, the decomposition $\mathcal{G}_{BB_{\Gamma}}$ is an edge-elliptic abelian splitting. We now show that $T_{BB}$ is universally elliptic and dominates every other $(\mathcal{A},\mathcal{H})$-tree.

Assume that $BB_{\Gamma}$ acts on another $(\mathcal{A},\mathcal{H})$-tree $T$ such that each generator acts elliptically. By Lemma~\ref{H_K is elliptic for any clique K}, the edge stabilizer $BB_{K}$ of $T_{BB}$ fixes a point in $T$ for every separating clique $K$ in $\Gamma$. Thus, the tree $T_{BB}$ is universally elliptic.

Let $G$ be an elliptic subgroup of $BB_{\Gamma}$. Then $G$ either fixes a vertex or an edge of $T_{BB}$. Let $BB_{\Gamma}$ act on another universally elliptic $(\mathcal{A},\mathcal{H})$-tree $T$ such that each generator acts elliptically. If $G$ fixes an edge in $T_{BB}$, then $G$ fixes a point in $T$ since $T_{BB}$ is universal elliptic. If $G$ fixes a vertex in $T_{BB}$, that is, $G$ is a vertex stabilizer of $T_{BB}$, then $G=BB_{\Gamma'}$ for some induced subgraph $\Gamma'$ of $\Gamma$ that is connected and has no separating cliques. If $\Gamma'$ is a clique, then $G$ fixes a point in $T$ by Lemma~\ref{H_K is elliptic for any clique K}. If $\Gamma'$ is not a clique and $G$ does not act elliptically on $T$, then $G$ must have an edge-elliptic abelian splitting. This against the construction of $\mathcal{G}_{BB_{\Gamma}}$. Thus, the group $G$ acts elliptically on $T$. Hence, the tree $T_{BB}$ dominates $T$. 
\end{proof}

We end this section with an example.

\begin{example}\label{JSJ example}
Let $\Gamma$ be the graph shown on the left hand side of Figure~\ref{example graph}.

\begin{figure}[H]
\begin{tikzpicture}[scale=0.6]
\draw [thick] (-1,0)--(0,2)--(1,0);
\draw [red, thick] (-1,0)--(0,-2)--(1,0)--(-1,0);
\draw [thick] (1,0)--(2,-2)--(0,-2);
\draw [thick] (-1,0)--(-2,-2)--(0,-2);

\draw [fill] (0,2) circle [radius=0.1];
\draw [fill] (-1,0) circle [radius=0.1];
\draw [fill] (1,0) circle [radius=0.1];
\draw [fill] (-2,-2) circle [radius=0.1];
\draw [fill] (0,-2) circle [radius=0.1];
\draw [fill] (2,-2) circle [radius=0.1];

\begin{scope}[shift={(4,0)}]
\draw [thick] (5.5,-2)--(7.5,-2)--(6.5,0)--(5.5,-2);
\draw [fill] (5.5,-2) circle [radius=0.1];
\draw [fill] (7.5,-2) circle [radius=0.1];
\draw [fill] (6.5,0) circle [radius=0.1];
\draw [thick] (8,-2)--(7,0)--(9,0)--(8,-2);
\draw [fill] (8,-2) circle [radius=0.1];
\draw [fill] (7,0) circle [radius=0.1];
\draw [fill] (9,0) circle [radius=0.1];
\draw [thick] (8.5,-2)--(10.5,-2)--(9.5,0)--(8.5,-2);
\draw [fill] (8.5,-2) circle [radius=0.1];
\draw [fill] (10.5,-2) circle [radius=0.1];
\draw [fill] (9.5,0) circle [radius=0.1];
\draw [thick] (7,0.5)--(9,0.5)--(8,2.5)--(7,0.5);
\draw [fill] (7,0.5) circle [radius=0.1];
\draw [fill] (9,0.5) circle [radius=0.1];
\draw [fill] (8,2.5) circle [radius=0.1];
\end{scope}
\end{tikzpicture}
\vspace*{5mm}
\caption{The figure on the left is the graph $\Gamma$. Each of the red edges is a minimal separating clique. The figure on the right illustrates the construction of the abelian splittings for $A_{\Gamma}$ and $BB_{\Gamma}$.}
\label{example graph}
\end{figure}

There are three separating $2$-cliques (the minimal size) in $\Gamma$, as shown in the left picture of Figure~\ref{example graph}. Each of these separating $2$-cliques gives an edge group $\mathbb{Z}^{2}$ in $\mathcal{G}_{A_{\Gamma}}$. After cutting along the separating $2$-cliques (the right picture in Figure~\ref{example graph}), there are four triangles and each of them determines a vertex group $\mathbb{Z}^{3}$ in $\mathcal{G}_{A_{\Gamma}}$. The left hand side of Figure~\ref{vertex-elliptic and edge-elliptic abelian JSJ-decompositions of example graph} shows a vertex-elliptic abelian JSJ-decomposition $\mathcal{G}_{A_{\Gamma}}$ of $A_{\Gamma}$.

Replace the vertex groups and edge groups of $\mathcal{G}_{A_{\Gamma}}$ by their associated Bestvina--Brady groups, that is, replace $\mathbb{Z}^{3}$ by $\mathbb{Z}^{2}$ and $\mathbb{Z}^{2}$ by $\mathbb{Z}$. Then we obtain an edge-elliptic abelian JSJ-decomposition $\mathcal{G}_{BB_{\Gamma}}$ of $BB_{\Gamma}$; see the right hand side of Figure~\ref{vertex-elliptic and edge-elliptic abelian JSJ-decompositions of example graph}.

\begin{figure}[H]
\begin{center}
\begin{tikzpicture}[scale=0.6]

\draw [thick] (0,2)--(0,0);
\draw [thick] (0,0)--(-2,-2);
\draw [thick] (0,0)--(2,-2);

\draw [fill] (0,2) circle [radius=0.1];
\draw [fill] (0,0) circle [radius=0.1];
\draw [fill] (-2,-2) circle [radius=0.1];
\draw [fill] (2,-2) circle [radius=0.1];

\node [above] at (0.1,2) {\tiny $\mathbb{Z}^{3}$};
\node [below] at (0,-0.3) {\tiny $\mathbb{Z}^{3}$};
\node [below left] at (-2,-2) {\tiny $\mathbb{Z}^{3}$};
\node [below right] at (2,-2) {\tiny $\mathbb{Z}^{3}$};

\node [right] at (0,1) {\tiny $\mathbb{Z}^{2}$};
\node [left] at (-0.9,-0.9) {\tiny $\mathbb{Z}^{2}$};
\node [right] at (1,-0.9) {\tiny $\mathbb{Z}^{2}$};

\begin{scope}[shift={(11,0)}]

\draw [thick] (0,2)--(0,0);
\draw [thick] (0,0)--(-2,-2);
\draw [thick] (0,0)--(2,-2);

\draw [fill] (0,2) circle (3pt);
\draw [fill] (0,0) circle (3pt);
\draw [fill] (-2,-2) circle (3pt);
\draw [fill] (2,-2) circle (3pt);

\node [above] at (0.1,2) {\tiny $\mathbb{Z}^{2}$};
\node [below] at (0,-0.3) {\tiny $\mathbb{Z}^{2}$};
\node [below left] at (-2,-2) {\tiny $\mathbb{Z}^{2}$};
\node [below right] at (2,-2) {\tiny $\mathbb{Z}^{2}$};

\node [right] at (0,1) {\tiny $\mathbb{Z}$};
\node [left] at (-0.9,-0.9) {\tiny $\mathbb{Z}$};
\node [right] at (1,-0.9) {\tiny $\mathbb{Z}$};
\end{scope}
\end{tikzpicture}
\end{center}
\caption{The left hand side is a vertex-elliptic abelian JSJ-decomposition of $A_{\Gamma}$; the right hand side is an edge-elliptic abelian JSJ-decomposition of $BB_{\Gamma}$.}
\label{vertex-elliptic and edge-elliptic abelian JSJ-decompositions of example graph}
\end{figure}
\end{example}

\begin{remark}
The Bestvina--Brady group $BB_\Gamma$ in Example~\ref{JSJ example} is $\mathrm{CAT}(0)$, and therefore, has quadratic Dehn function. However, Papadima and Suciu showed that this $BB_\Gamma$ is not isomorphic to any RAAG; see \cite[Proposition 9.4]{papadimaandsuciualgebraicinvariantsforbestvinabradygroups}
\end{remark}

\section*{acknowledgements}

The author thanks Pallavi Dani and Tullia Dymarz for their constant support. The author thanks Michael Hull for bringing Zaremsky's work to his attention. The author is grateful for the referee's valuable comments and suggestions. 

\bibliographystyle{plain}
\bibliography{20220925}
\end{document}